\documentclass[12pt]{amsart}
\usepackage{graphicx} 
\usepackage{inputenc,amsfonts,amsthm,amscd,enumerate,amsmath, txfonts, comment, cleveref}

\usepackage{xcolor}
\numberwithin{equation}{section}

\allowdisplaybreaks

\newtheorem{theorem}{Theorem}[section]
\newtheorem{corollary}[theorem]{Corollary}
\newtheorem{lemma}[theorem]{Lemma}

\newtheorem{remark}[theorem]{Remark}
\newtheorem{proposition}[theorem]{Proposition}

\newtheorem{definition}[theorem]{Definition}

\newcommand{\Hbx}[0]{\left|\text{Hess}_{b_x^2}-\frac{\Delta b_x^2}{n}g\right|}

\newcommand{\Hubx}[0]{\left|\text{Hess}_{\underline{b}_x^2}-\frac{\Delta \underline{b}_x^2}{n}g\right|}

\newcommand{\R}[0]{\mathbb R}

\newcommand{\tv}[0]{\mathrm{v}_M}

\newcommand{\Brp}{B_p(r)}
\newcommand{\Brpi}{B_{p_i}(r)}

\title[Quantitative Rigidity]{Quantitative Rigidity Using Colding's Monotonicity Formulas for Ricci Curvature}

\thanks{
CB supported in part by NSF DMS CAREER-1750254. JP was supported by the National Research Foundation of Korea (NRF) grant funded by the Korea government (MSIT) RS-2024-00346651. }

\author[Breiner]{Christine Breiner}
\address{Brown University\\
Department of Mathematics\\
Providence, RI, USA}
\email{christine\underline{ }breiner@brown.edu}

\author[Park]{Jiewon Park}
\address{Korea Advanced Institute of Science and Technology (KAIST)\\
Department of Mathematical Sciences\\
Daejeon, South Korea}
\email{jiewonpark@kaist.ac.kr}

\begin{document}

\maketitle

\begin{abstract}
In \cite{Colding}, Colding proved monotonicity formulas for the Green function on manifolds with nonnegative Ricci curvature. Inspired by the sharp estimates relating the pinching of monotone quantities to the splitting function in \cite{cjn}, in this paper we investigate quantitative control obtained from pinching of Colding's monotone functionals. From the Green functions with poles at $(k+1)$-many independent points, $k$-splitting functions are constructed with regularity quantitatively controlled by the pinching. Moreover, the pinching at these independent points controls the distance to the nearest cone of the form $\R^k \times C(X)$. 
\end{abstract}

\section{Introduction}

The investigation of splitting phenomena on Riemanniann manifolds with non-negative Ricci curvature has a long history. Originating with the work of Cohn-Vossen \cite{cv} in dimension two, it was later extended by Toponogov \cite{t} to higher dimensions for manifolds with non-negative sectional curvature. In  
the celebrated work \cite{cg}, Cheeger and Gromoll prove rigidity for non-negative Ricci in all dimensions. Namely, given a geodesically complete connected Riemannian manifold $(M^n,g)$ with $\mathrm{Ric} \geq 0$, if $M$ contains a geodesic line then $M$ splits isometrically as a product $N \times \mathbb R$. In \cite{cheeger-colding}, Cheeger and Colding extend this rigidity result to almost rigidity, demonstrating that if there exist balls with ``almost geodesic lines" then these balls are Gromov-Hausdorff close to a ball in a product space $X \times \mathbb R$.

Using monotone quantities to prove rigidity results is a common theme in differential geometry and such results are now a crucial first step in improving regularity results. Indeed, starting with almost rigidity statements and using the quantitative stratification techniques introduced by \cite{chen3,chen}, and later refined by \cite{nv,jn}, provides a natural framework for improved regularity. Recently,
Cheeger, Jiang, and Naber \cite{cjn} adapted this framework to a monotone quantity they call the \emph{local pointed entropy} to prove that the $k$-th stratum, $\mathcal S^k$, of the singular set of a Ricci limit space is $k$-rectifiable and for a.e. $x \in \mathcal S^k$, every tangent cone at $x$ is $k$-symmetric. Key to their arguments is establishing the existence and regularity of the so called \emph{almost splitting functions} which play the role of the coordinate functions in the almost rigidity statements. The regularity result is encoded in a sharp splitting estimate (cf. \cite[Theorem 6.1]{cjn}) which controls the behavior of these almost splitting functions by the pinching of the local pointed entropy.

 Colding and Minicozzi \cite{cm-family}, extending the work of 
Colding \cite{Colding}, determine a family of new monotone quantities for Riemannian manifolds with non-negative Ricci curvature using a power of the Green function $G$, namely $b_x=G(x,\cdot)^{1/(2-n)}$, and these quantities have already been crucial in establishing regularity results. For example, by determining a sufficiently fast decay rate for the derivative of  one such monotone quantity,  Colding and Minicozzi  \cite{cm-einstein} prove the uniqueness of tangent cones at infinity for Einstein manifolds (presuming one tangent cone has smooth cross section). 
(See also Gigli and Violo \cite{gigli-violo} for analogous monotonicity formulas on RCD spaces and rigidity and almost rigidity results therein.)

This paper was motivated by a desire to understand how Colding's monotonicity formulas might be used to establish the existence and quantitatively control the regularity of almost splitting functions by the pinching of said monotone quantities. Indeed, since $b_x$ is asymptotic to (a multiple of) the distance $d_x$ from $x$, it seems natural to presume that a splitting function could be built from $b_{x_1}^2-b_{x_2}^2$ where $x_1, x_2$ are two distinct points, and that the function should be controlled by monotone quantities involving $b_{x_1}, b_{x_2}$. This is precisely what we do in Theorems \ref{thm:k splitting} and Theorem \ref{thm:k splitting by pinching only}.  In other words, small pinching implies the existence of an almost $k$-splitting map with regularity controlled in a quantitative way.

We hope that the results here can provide an introduction into the study of splitting through Colding's monotone quantities and in the future might prove useful in improving regularity results for tangent cones at infinity for non-negative Ricci curvature.

\subsection{Main Results}

Due to the technical nature of the theorem statements, we give a less precise presentation of the statements here. For the precise statement of the results, we refer the reader to the theorem statements in section \ref{sec: splitting}.

We consider $(M^n,g)$, $n \geq 3$, a Riemannian manifold with $\mathrm{Ric} \geq 0$ and Euclidean volume growth, i.e., $\lim_{R \to \infty}\frac{Vol(B_p(R))}{R^n}=:\tv>0$. The monotone quantities of interest are defined in Definition \ref{def:A_x} and agree with those in \cite{Colding} up to a renormalization of $b_x$ so that in this paper $\lim_{d_x(y) \to \infty} \frac{b_x(y)}{d_x(y)}=1$ (see Definition \ref{def:b_x}). In particular, $F_x, A_x, V_x$ are all monotone where
    \begin{align*}
     F_x(r) &:= (A_x - 2(n-1)V_x)(r),
\\  A_x(r)&:= r^{1-n}\int_{b_x = r} |\nabla b_x|^3\,dA,\\
    V_x(r)&:=r^{-n}\int_{b_x \le r} |\nabla b_x|^4 \,dV.
    \end{align*}
    We denote the corresponding pinched quantities by
    \begin{align*}
\mathcal W_{s,t}(x)&:=|A_x(t)-A_x(s)|,\\
\mathcal F_{s,t}(x)&:=|F_x(t)-F_x(s)|.
\end{align*}
We first demonstrate that both of the above pinched quantities can be used to bound the scale invariant distance to appropriate domains on a cone. To do this, we first show that the $L^2$-norm of the trace-free Hessian of $b_x^2$ is controlled on sublevel sets of $b_x$ by $\mathcal F$ (see Proposition \ref{prop:0pinching_2}). Similarly, the pinching $\mathcal W$ controls a weighted $L^2$ norm on superlevel sets (see Proposition \ref{prop:0pinching}). Next, applying results of \cite{Colding} and \cite{cm-einstein}, we bound the scale invariant distance to a ball or an annulus in a cone; that is, in either case there exist metric cones $C(X),C(Y)$ with vertex $o$ such that 
$$\left(\frac{ d_{GH}(B^{M}_{x}(s), B^{C(X)}_{o}(s) )}s\right)^{2+2\mu(n)} \le C(n,\tv) \mathcal F_{2s,4s}(x), \text{ and}$$
$$\left(\frac{d_{GH}(A^{d^M}(x, s, 2s), A^{d^{C(Y)}} (o, s, 2s))}{s} \right)^{2+\mu(n)} \le C(n,\tv)\mathcal W_{\frac{s}{4}, \frac{s}{2}}(x).$$
(See Definition \ref{def:annulus} for the definition of the annular domains of interest.)

Building on the $L^2$ control on the trace-free Hessian by $\mathcal F$ given in Proposition \ref{prop:0pinching_2} and using the uniform $C^0$-estimates we establish in Proposition \ref{prop:UniformEst}, we demonstrate initial control on our splitting functions in Theorem \ref{thm:k splitting}.  Our result can be viewed as an analogue to \cite[Proposition 6.4 or Theorem 6.1]{cjn} with the global monotonicity formulas for the Green functions replacing their local monotone quantities.

\begin{theorem}[roughly Theorem \ref{thm:k splitting}]\label{thm:k splitting intro}
Let $r>0$ and $p \in M$. If 
\begin{itemize}
\item there exist $(k+1)$-points $\{x_0, \dots, x_k \} \subset B_p(r)$ which almost span a $k$-plane and 
\item for $s \gg r$ each $B_{x_i}(s)$ is an almost metric cone, 
\end{itemize} then there exists a $k$-splitting function $u:B_p(r) \to \R^k$ given roughly by $u_j:= \frac{b_{x_j}^2-b_{x_0}^2-d(x_j,x_0)^2}{2d(x_j,x_0)}$, $j=1, \dots, k,$ such that
\begin{itemize}
\item $r^2 \fint_{\Brp}|\text{Hess}_{u}|^2 \,dV $ is controlled by \[\sum_{j=0}^k\left(\left( \frac sr\right)^n \mathcal{F}_{s,2s}(x_j)+ \fint_{\Brp} \left| \text{Hess}_{b_{x_j}^2} - \frac{\Delta b_{x_j}^2}{n}g\right|^2 b_{x_j}^{-2} \, dV\right),
\]
     \item $\fint_{\Brp}\left|\langle \nabla u_i, \nabla u_j\rangle - \delta_{ij}\right|^2 \,dV\leq C \sum_{\ell=0}^{k} \mathcal F_{s,2s}(x_\ell)$ for all $i,j=1, \dots, k$, 
     \item $\sup_{\Brp}|\nabla u_i|\leq C$ for all $i = 1, \dots, k$.
\end{itemize}
\end{theorem}

If the pinching of $\mathcal F_{s,2s}(x_j)$ is sufficiently small relative to the scale ratio $s/r$ and the trace-free Hessian on $B_p(2r)$, we can say more. Indeed combining Theorem \ref{thm:k splitting} with Proposition \ref{prop:pinching bound}, which controls the weighted Hessian term in item (1) above by the pinching, we immediately obtain the following bound on the splitting functions at scale $r$ by the pinching $\mathcal F$ at scale $s$.

\begin{theorem}[roughly Theorem \ref{thm:k splitting by pinching only}]\label{thm:k splitting by pinching only intro}
Let $r>0$ and $p \in M$. If 
\begin{itemize} 
\item there exist $(k+1)$-points $\{x_0, \dots, x_k \} \subset B_p(r)$ which almost span a $k$-plane, 
\item for $s \gg r$ each $B_{x_i}(s)$ is an almost metric cone, 
\item and $\left(\frac sr \right)^n  \mathcal F_{s,2s}(x_i)$ is sufficiently small, for all $i=0\dots k$, depending on $n$ and $\sup_{B_p(2r)} \left|\text{Hess}_{b_x^2} - \frac{\Delta b_x^2}{n}g\right|^2(x)$,
\end{itemize} then the function $u:\Brp \to \R^k$ of Theorem \ref{thm:k splitting intro} satisfies the previous estimates and
\[r^2 \fint_{\Brp}|\text{Hess}_{u}|^2 \,dV \text{ is bounded from above by }\left(\frac{s}{r}\right)^n \sum_{j=0}^k\mathcal{F}_{s,2s} (x_j)^{1-2/n}.\]
\end{theorem}

The assumption in Theorem \ref{thm:k splitting by pinching only intro} that $B_{x_i}(s)$ is almost a metric cone is in fact somewhat redundant and can be reformulated purely in terms of the pinching $\mathcal F$. Indeed, if the pinching is small at large scales, then the ball is nearly a metric cone. This relationship is explained in Remark \ref{rem:symmetry}.

\begin{remark}
The sharp estimates in \cite[Theorem 6.1]{cjn} on the splitting function by the pinching of the local pointed entropy are not only interesting in their own right, but act as a main ingredient in the proof of $k$-rectifiability of the singular $k$-strata. However the proof of the rectifiability crucially requires not just the sharp estimate, but also the harmonicity of the splitting function in many places. One key instance is the nondegeneration theorem  \cite[Theorem 8.1]{cjn}. There a telescopic estimate for harmonic functions is used, which is stronger than general estimates for $W^{1,p}$ functions. In the telescopic estimate, harmonicity allows one to get rid of the term $\langle \nabla \Delta u_i, \nabla u_i \rangle$ arising in the Bochner formula.

In general the splitting functions constructed from the Green function in this paper need not be harmonic (see also remark \ref{rmk: not harmonic}). However, provided that one has a sufficiently sharp control on integrals of $|\nabla \Delta u_i|$, then it seems very likely that similar conclusions as \cite{cjn} can be deduced using these functions. It remains an interesting question under what geometric conditions these splitting functions, or (almost equivalently) the Green function, indeed satisfies such finer estimates.
\end{remark}

Once we have the $k$-splitting functions of Theorem \ref{thm:k splitting by pinching only intro}, by a standard compactness-contradiction argument we can immediately conclude an almost splitting theorem using the map $u$.

\begin{theorem}[roughly Theorem \ref{thm:rigidity}]\label{thm:rigidity intro}
Let $r>0$ and $p \in M$. If 
\begin{itemize} 
\item there exist $(k+1)$-points $\{x_0, \dots, x_k \} \subset B_p(r)$ which almost span a $k$-plane, 
\item for $s \gg r$ each $B_{x_i}(s)$ is an almost metric cone, 
\item and $\left(\frac sr \right)^n  \mathcal F_{s,2s}(x_i)$ is sufficiently small, for all $i=0\dots k$, depending on $n$ and $\sup_{B_p(2r)} \left|\text{Hess}_{b_x^2} - \frac{\Delta b_x^2}{n}g\right|^2(x)$,
\end{itemize} then there exist $C(Y)$, the cone over a metric measure space $Y$, and a function $\mathcal U:=(u,\tilde u): B_p(r) \to \R^k \times C(Y)$, where $u:B_p(r) \to \R^k$ is the function from Theorem \ref{thm:k splitting by pinching only intro}, such that  
\[
d_{GH}(B_p({r/2}), \mathcal U(B_p({r/2})))<\epsilon r.
\]
\end{theorem}

\subsection{Why $F_x$?}
An explanation of why we opt to work with $F_x=A_x-2(n-1)V_x$ to get estimates on our splitting functions, rather than $A_x$ or even $V_x$, is in order. First, while ${V}_x$ itself is monotone nonincreasing (\cite[Corollary 2.19]{Colding}), this follows from observing that $0 \le {A}_x' \le 2(n-1){V}_x'$. The direct calculation of ${V}_x'$ does not seem to have an expression like (\ref{underline A'}). But the question remains, why not use $A_x$ to get estimates on our splitting functions? There are two distinct reasons for this, both relating to the fact that $\mathcal W_{s,2s}(x):=|A_x(2s)-A_x(s)|$ only controls the trace free Hessian of $b_x^2$ on a superlevel set of $b_x$ rather than a sublevel set (see Lemma \ref{lem:Aprime}).

Recall that our goal is to produce splitting functions using the functions $b_x$, and to do this we will need uniform estimates on $b_x$ in the domain of the splitting functions. The necessary uniform estimates come from Proposition \ref{prop:UniformEst}, and in its proof two things are required. First, we need a fixed inner scale $r$ on which we will get our estimates and a much larger scale $s$ on which the almost $0$-symmetry holds. By Corollary \ref{cor:symmetry} and Remark \ref{rem:symmetry}, pinching of $F_x$ implies almost $0$-symmetry on a ball whereas pinching of $A_x$ implies almost $0$-symmetry on an annulus. So a result like Corollary \ref{cor:uniform est} cannot hold by assuming $\mathcal W_{s,2s}(x)<\delta$ since the limiting objects in the proof of Proposition \ref{prop:UniformEst} might not be cones all the way down to the vertex but instead will be conical only on an annulus at scale $s$.

Second, and perhaps more importantly, even if we added the hypothesis of $(0,\delta)$-symmetry for $B_x(s)$ and try to get estimates like those of Theorem \ref{thm:k splitting} with $\mathcal W_{s,2s}(x)$ replacing $\mathcal F_{s,2s}(x):=|F_{x}(2s)-F_{x}(s)|$, the two scales $r,s$ pose a problem. The pinching estimate of Proposition \ref{prop:0pinching_2} is conducive to these two scales since clearly $B_p(r) \subset B_{x}(s)$ for $s$ much larger than $r$. On the other hand, Proposition \ref{prop:0pinching} cannot be modified to have two different scales since, for example $A^b(p,r,2r)\not\subset A^b(p,s,2s)$.

\subsection{Outline of the paper}
This paper is organized as follows. In section \ref{sec: prelim} we define the relevant terms, recall a necessary Poincar\'e inequality, and describe our scaling convention for the Green function. In section \ref{sec: cone}  we first prove that small $\mathcal W$-pinching controls the scale invariant distance of an annulus to an annulus in a cone (Proposition \ref{prop:0pinching} and Corollary \ref{cor:distance to cone}). Next, we prove that small $\mathcal F$-pinching implies almost 0-symmetry (Proposition \ref{prop:0pinching_2} and Corollary \ref{cor:symmetry}). We then establish crucial uniform estimates on $b$ on small balls when large balls are almost 0-symmetric (Proposition \ref{prop:UniformEst} and Corollary \ref{cor:uniform est}). Finally, in section \ref{sec: splitting} we prove our main results. First, we show that small $\mathcal F$-pinching at $(k,\alpha)$-independent points implies the almost $k$-splitting of Theorem \ref{thm:k splitting intro} (cf. Theorem \ref{thm:k splitting}). We then prove that small $\mathcal F$-pinching relative to the scale and the trace-free Hessian on a much smaller scale implies $\mathcal F$-pinching also controls a weighted $L^2$-norm on the trace-free Hessian (cf. Proposition \ref{prop:pinching bound}). Theorems \ref{thm:k splitting by pinching only intro} and \ref{thm:rigidity intro} are then stated precisely (cf. Theorems \ref{thm:k splitting by pinching only} and \ref{thm:rigidity}), the results of which are immediate.

\section{Preliminaries} \label{sec: prelim}

For the rest of this paper, we presume that $(M^n,g)$ satisfies $\mathrm{Ric}\geq 0$ and $n \geq 3$. We use the standard notation for balls, i.e. for $x \in M$ and $d_x$ the distance from $x$,
\[
B_x(r):=\{y \in M: d_x(y)<r\}.
\] We also presume that
\[
 \lim_{R \to \infty}\frac{Vol(B_p(R))}{R^n}=:\tv\geq \mathrm v
\]
for a fixed constant $\mathrm v>0$. Note that there are many propositions where the uniform lower bound $\mathrm v$ on $\tv$ is unnecessary and we only require $\tv>0$, but in this case we may just consider $\mathrm v=\tv$. Note also that if one has $\mathrm v_1 \geq \mathrm v_2$, then any constants that arise depending on the volume growth can be taken so that $C(\mathrm v_1) \leq C(\mathrm v_2)$.

\subsection{Quantitative definitions}

Following the standard in the literature, we define quantitative independence and quantitative symmetry. Notice that by definition both are scale invariant quantities.
\begin{definition}
Let $(X,d)$ be a metric space and let $\alpha>0$. Fix a ball $B_{x}(r) \subset X$. A set of points $U=\{x_0, \dots, x_k\}\subset B_{x}(r)$ is called \emph{$(k, \alpha)$-independent} if for any $U'=\{x_0', \dots, x_k'\} \subset \R^{k-1}$,
\[
r^{-1}d_{GH}(U,U') \geq \alpha.
\]
\end{definition}
\begin{definition}
Let $(X,d)$ be a metric space and let $\epsilon>0$. We say a ball $B_x(r)\subset X$ is \emph{$(k, \epsilon)$-symmetric} if there exists a $k$-symmetric metric cone $X'=\R^k \times C(Z)$ with vertex $x'$ such that
\[
r^{-1}d_{GH}(B_x^X(r), B_{x'}^{X'}(r))< \epsilon.
\]
\end{definition}
\subsection{Poincar\'e inequality}

The following result is immediate from \cite[Theorem 1]{hk}. 

\begin{lemma}[Neumann Poincar\'e inequality]\label{lem:NPi}
   Let $\Omega \subset M$ be smooth and connected with smooth boundary. There exist $k>1$ and $C=C(n,\tv)>0$ so that if $u \in W^{1,2}(\Omega)$, then
    \begin{equation}
       \left( \fint_\Omega |u-\fint_\Omega u|^{2k} \, dV\right)^{1/k}\leq C (\mathrm{diam}(\Omega))^2 \fint_\Omega |\nabla u|^2\, dV. 
  \end{equation}
\end{lemma}
In our setting, when $\Omega$ is smooth and connected it immediately satisfies the so called ``$C(1,m)$" condition of \cite{hk}. Buser's inequality \cite{buser} or \cite{cheeger-colding}, using the segment inequality, imply that for any ball $B \subset \Omega$ of radius $r$,
\[
\fint_B|u-\fint_B u|\, dV \leq C(n) r \fint_B |\nabla u|\,dV.
\]As Cauchy-Schwarz implies
\[
\fint_B |u-\fint_B u|\,dV \leq C(n) r \left(\fint_B |\nabla u|^2\,dV\right)^{1/2},
\]all of the hypotheses of \cite[Theorem 1]{hk} are immediately satisfied and the result follows.

\subsection{Green functions, scaling conventions}\label{sec:greens_scaling} Following \cite{cm-schroedinger}, we let
\begin{equation}\label{eq:binfty}
    b_\infty:=\left( \frac{\tv}{\omega_n}\right)^{1/(n-2)}
\end{equation}
where $\omega_n$ denotes the volume of the unit ball in $\mathbb R^n$.  
As all spaces we consider are of Euclidean volume growth, we can choose the Green function to be the unique minimal positive one, normalized in the same way as \cite{cm-schroedinger}, where the existence and uniqueness of this function follows from Li-Tam theory (see for example \cite{LT}).
To clarify the different normalization constants across the literature, note that
\cite{cm-schroedinger} defines the Green function (which we denote $G^{CM}_x$) to satisfy 
\[
\Delta G^{CM}_x:= -n(n-2)\omega_n \delta_x,
\]while \cite{ding} defines the Green function (which we denote $G^D_x$) so that
\[
\Delta G_x^D:=-\delta_x.
\]Therefore
\[
G_x^{CM}= n(n-2) \omega_n G_x^D.
\]
In Proposition \ref{prop:UniformEst}, we will need \cite[Corollary 4.22]{ding} which says that for a limit cone of a sequence $(M_i,g_i,p_i)$ with $\text{Ric}_{M_i} \geq 0$ and $\mathrm v_{M_i}\geq\mathrm v>0$, there exists a unique Green function on this limit cone $C(Y)$ such that
\[
G^D_\infty(p_\infty,x)= \frac{d_\infty^{2-n}(p_\infty,x)}{(n-2)\mu_Y(Y)}.
\]Here $p_\infty$ is the vertex of the limit cone and $\mu_Y$ is the natural measure induced by restricting the limit measure to $Y$. It will be helpful to note that Ding's result and the scaling relations imply that
\[
G^{CM}_\infty(p_\infty,x)= \frac{n \omega_n}{\mu_Y(Y)}{d_\infty^{2-n}(p_\infty,x)},
\]and since $\mu_Y(Y)= n\lim_{i \to \infty}\mathrm v_{M_i}$, 
\[
G^{CM}_\infty(p_\infty,x)= b_\infty^{2-n} d_\infty^{2-n}(p_\infty,x).
\]

\begin{definition} \label{def:b_x}

Let $$ b_x:= b_\infty^{-1} \underline b_x:M \to \R$$ where $\underline b_x(\cdot):= G^{CM}_x(\cdot)^{\frac 1{2-n}}$. 
\end{definition}
By \cite{cm-schroedinger}, 
\begin{equation}\label{asymp:G}
 \lim_{d(x,y) \to \infty}
\left|\frac{G_x^{CM}(y)}{\left({b_\infty}d(x,y)\right)^{2-n}} - 1\right| = 0.   
\end{equation} 
So (\ref{asymp:G}) implies that for $b_x$ as in Definition \ref{def:b_x},
\begin{equation} \label{eqn: CM b conv}
\lim_{d(x,y) \to \infty} \left|\frac{b_x(y)}{d(x,y)} - 1\right| = 0.
\end{equation}
In particular, $b_x = d_x:=d(x,\cdot)$ on a cone with vertex $x$.

Note that \cite[Theorem 3.1]{Colding} demonstrates the pointwise gradient estimate that $|\nabla \underline{b}_x| \leq 1$ with equality at a point implying isometry to the Euclidean space. For $b_x$, this becomes
\begin{equation}\label{eq:bgradbound}
    |\nabla b_x| \le b_\infty^{-1}. 
\end{equation}

For convenience we use the following notation for annuli defined in terms of level sets of $b_x$ and $d_x$, respectively.
\begin{definition}\label{def:annulus}Define
$$A^b(x,a,c):= \{y \in M: a \leq b_x(y) \leq c\},$$ 
and\[
A^d(x,a,c):=\{y \in M: a \leq d_x(y) \leq c\}.
\]

\end{definition}

Finally we compare sublevel sets of $b_x$ and $d_p$ for $d(x,p) <r$. 

\begin{lemma}\label{lem:subball}For any $r>0$ and $x \in B_p(r)$, 
\begin{equation}\label{eq:subball}
B_p(r) \subset \{b_x \le 2b_\infty^{-1}r\}.
   \end{equation}
\end{lemma}

\begin{proof} 
For $y \in \Brp$, by (\ref{eq:bgradbound}), \[b_x(y) \le b_{\infty}^{-1} d(x,y) \le b_{\infty}^{-1} \left( d(x,p) + d(p,y) \right) \le 2b_\infty^{-1}r.\]
\end{proof}

\section{Estimates on the almost cone-splitting function $b$} \label{sec: cone}

\subsection{Pinched quantities}
We begin this section by describing the monotone quantities of interest and set the definition for the \emph{pinching}.

\begin{definition} \label{def:A_x}
    Following \cite{Colding}, but using our new normalization for $b_x$, we let
    \[
    A_x(r):= r^{1-n}\int_{b_x = r} |\nabla b_x|^3\,dA,
    \]
      \[
    V_x(r):=r^{-n}\int_{b_x \le r} |\nabla b_x|^4 \,dV,
    \]and
    \[
    F_x(r) := (A_x - 2(n-1)V_x)(r).
    \]
Finally, let
\[
\mathcal W_{s,t}(x):=|A_x(t)-A_x(s)|
\]and
\[
\mathcal F_{s,t}(x):=|F_x(t)-F_x(s)|.
\]
 
\end{definition}

For $\underline{b}_x$ as in Definition \ref{def:b_x} and $\underline{A}_x(r) := r^{1-n}\int_{\underline{b}_x = r} |\nabla \underline{b}_x|^3d\mathrm{Area}$, $\underline{V}_x(r) := r^{-n}\int_{\underline{b}_x \le r} |\nabla \underline{b}_x|^4 dV$, Colding proves \cite[Corollary 2.21 and Theorem 2.4]{Colding}
\begin{equation} \label{underline A'}
  \underline{A}_x'(r)= -\frac{r^{n-3}}2\int_{\underline{b}_x \geq r} \left(\Hubx^2 + \text{Ric}(\nabla \underline{b}_x^2, \nabla \underline{b}_x^2)\right) \underline{b}_x^{2-2n}\, dV,
\end{equation}
and
\begin{equation} \label{underline F'}
\left(\underline{A}_x-2(n-1)\underline{V}_x\right)'(r) = \frac{r^{-1-n}}{2} \int_{\underline{b}_x \leq r} \left(\Hubx^2 + \text{Ric}(\nabla \underline{b}_x^2, \nabla \underline{b}_x^2)\right) \, dV.    
\end{equation}

With our normalization $b_x = b_\infty^{-1} \underline{b}_x$, (\ref{underline A'}) and (\ref{underline F'}) translate into

\begin{lemma}[Monotonicity of $A$ and $F$]\label{lem:Aprime}
    \[
   A_x'(r)= -\frac{r^{n-3}}2\int_{b_x \geq r} \left(\Hbx^2 + \mathrm{Ric}(\nabla b_x^2, \nabla b_x^2)\right) b_x^{2-2n}\, dV,
    \] 
    \[
   F_x'(r)= \frac{r^{-1-n}}{2} \int_{b_x \leq r} \left(\Hbx^2 + \mathrm{Ric}(\nabla b_x^2, \nabla b_x^2)\right) \, dV.
    \]
\end{lemma}

Using the monotonicity of $A_x$, we next deduce a quantitative bound on a weighted norm of the trace-free Hessian of $b_x^2$ on an annulus at scale $s$ in terms of $\mathcal W_{s,2s}(x)$.

\begin{proposition}\label{prop:0pinching} 
For any $x \in M$, $s>0$ and $C>2$, 
\[
\int_{b_x \geq 2s} \Hbx^2b_x^{-n} \, dV \leq 2C^{n-2}\mathcal W_{s,2s}(x).
\]In particular
\begin{equation}\label{eq:tracefreebound1}
\int_{A^b(x,2s,Cs)} \Hbx^2b_x^{-n} \, dV\leq 2C^{n-2}\mathcal W_{s,2s}(x).
\end{equation}
    \end{proposition}
  \begin{proof}  
By Lemma \ref{lem:Aprime} and the mean value theorem
there exists some $r \in [s, 2s]$ such that 
\[
\frac{sr^{n-3}}2 \int_{b_x \geq r} \left(\Hbx^2 + \text{Ric}(\nabla b_x^2, \nabla b_x^2)\right) b_x^{2-2n}\, dV  = \mathcal W_{s,2s}(x).
\]Since $s \leq r \leq 2s$, 
\begin{equation*}
\frac{s^{n-2}}2 \int_{b_x \geq 2s} \left(\Hbx^2 + \text{Ric}(\nabla b_x^2, \nabla b_x^2)\right) b_x^{2-2n}\, dV  \leq \mathcal W_{s,2s}(x).
\end{equation*} As $A^b(x, 2s, Cs) \subset \{b_x \geq 2s\}$, \eqref{eq:tracefreebound1} follows immediately.
\end{proof}

Combining the above with \cite[Equation 2.55]{cm-einstein}, the proof of which is outlined in \cite[Appendix C]{cm-einstein}, we confirm an expectation of Colding and Minicozzi that the pinching $\mathcal W$ of $A$ controls the scale-free distance of an annulus in $M$ to an annulus in a cone; see \cite[Equation 1.19]{cm-einstein} and the remark following. 
\begin{corollary}\label{cor:distance to cone} 
There exist $C(n,\tv)>0$, $\mu(n)>0$, and $s_0(n,\tv)>0$ so that for any $x \in M$, whenever $s>s_0$, there exists a metric cone $C(X)$ with vertex $o$ such that
$$\left(\frac{d_{GH}(A^{d^M}(x, s, 2s), A^{d^{C(X)}} (o, s, 2s))}{s} \right)^{2+\mu(n)} \le C(n,\tv)\mathcal W_{\frac{s}{4}, \frac{s}{2}}(x).$$ 
\end{corollary}
Note that the cone $C(X)$ a priori depends on $x$ and $s$, and no uniqueness of limit is declared here.

We next show that an analogue to Proposition \ref{prop:0pinching} holds using the function $F_x$ rather than $A_x$. One advantage of $F$-pinching over $A$-pinching is that $F$-pinching controls the trace-free Hessian on a ball, rather than on an annulus. (For a more detailed discussion on why $F$-pinching is advantageous in the construction of splitting functions, see comments in the Introduction.) 
\begin{proposition}\label{prop:0pinching_2} 
For any $x \in M$ and $r_2>r_1>0$,
\[
 \int_{b_x \leq r_1} \left(\Hbx^2 + \mathrm{Ric}\left(\nabla b_x^2, \nabla b_x^2\right)\right)\, dV  \leq 2r_2^{n+1}\frac{\mathcal F_{r_1,r_2}(x)}{r_2 - r_1}.
\]
In particular, for any $C>1$ and $r>0$,
\begin{equation}\label{eq:tracefreebound}
r^{-n} \int_{b_x \le r} \Hbx^2 \, dV\leq \frac{2 C^{n+1}}{C-1} \mathcal F_{r,Cr}(x).
\end{equation}
    \end{proposition}
  \begin{proof}  
Recall that by Lemma \ref{lem:Aprime}, $$  F_x'(r) = \frac{r^{-1-n}}{2} \int_{b_x \leq r} \left(\Hbx^2 + \text{Ric}(\nabla b_x^2, \nabla b_x^2)\right) \, dV.$$

By the mean value theorem, there exists some $s \in [r_1, r_2]$ such that 
\[
\frac{s^{-1-n}}{2} \int_{b_x \leq s} \left(\Hbx^2 + \text{Ric}(\nabla b_x^2, \nabla b_x^2)\right)\, dV  = \frac{F_x(r_2) - F_x(r_1)}{r_2 - r_1}.
\]
This implies that
\[
\frac{(r_2)^{-1-n}}{2} \int_{b_x \leq r_1} \left(\Hbx^2 + \text{Ric}(\nabla b_x^2, \nabla b_x^2)\right)\, dV  \leq \frac{F_x(r_2) - F_x(r_1)}{r_2 - r_1}
\]which gives the first result. Now, 
taking $r_1 = r$ and $r_2 = Cr_1$ for $C>1$, we get as promised,
\[
r^{-n} \int_{b_x \leq r} \left(\Hbx^2 + \text{Ric}(\nabla b_x^2, \nabla b_x^2)\right)\, dV  \leq \frac{2 C^{n+1}}{C-1} \left(F_x(Cr) - F_x(r)\right).
\]
\end{proof}

As a corollary, and using \cite[Corollary 4.8]{Colding}, we control the scale invariant distance to a cone by the $\mathcal F$-pinching. 

\begin{corollary}\label{cor:symmetry}
There exist $C(n,\tv)>0$, $\mu(n)>0$, and $s_0(n,\tv)>0$ so that for any $x \in M$, whenever $s>s_0$, there exists a cone $C(X)$ with vertex $o$ such that
$$\left(\frac{ d_{GH}(B^{M}_{x}(s),B^{C(X)}_{o}(s) )}s\right)^{2+2\mu(n)} \le C(n,\tv) \mathcal F_{2s,4s}(x).$$
\end{corollary}
\begin{proof}
By  \cite[Corollary 4.8]{Colding}, there exist constants $C(n,\tv), \mu(n),$ and $s_0(n,\tv)>0$ so that whenever $s>s_0$, there exists a cone $C(X)$ such that
$$\left(\frac{d_{GH}(B^{M}_{x}(s),B^{C(X)}_{o}(s) )}s\right)^{2+2\mu} \le C s^{-n} \int_{b_x \le 2s} \Hbx^2 \, dV.$$  The result then follows immediately from Proposition \ref{prop:0pinching_2}.
    
\end{proof}

\begin{remark} \label{rem:symmetry} 
Note that the corollary above immediately implies that if $\mathcal{F}_{2s, 4s}(x) \le \delta^{4+4\mu}/C$, then $B_x(s)$ is $(0,\delta^2)$-symmetric. Here $C$ and $\mu$ are the constants from Corollary \ref{cor:symmetry}. 
\end{remark}

\subsection{Uniform estimates on $b_x$}
By (\ref{eqn: CM b conv}) we already know that $b_x/d_x$ converges to $1$ as $d_x \to \infty$. In the following proposition we show that at points which are almost conical at large scales, the difference is uniformly bounded on fixed smaller scales. Additionally, we get a uniform integral gradient estimate.

\begin{proposition}\label{prop:UniformEst} [Uniform $\mathcal{C}^0$- and $W^{1,2}$-convergence of $b_x$ on fixed balls]
Fix $r>0$, $p \in M$ and $C>1$. Given $\epsilon>0$, there exists $\delta_0=\delta_0(\epsilon, n, \mathrm v)>0$ such that the following holds. For any $0<\delta\le\delta_0$ and any point $x \in B_p(r)$ such that $B_x(1/\delta)$ is $(0,\delta^2)$-symmetric, we have
\[
\sup_{B_p(r)}\left| \frac{b_{x}^2}{d_{x}^2}-1\right|<\epsilon
\]and 
\[
\left|\fint_{B_p(r)}(\left| \nabla b_{x} \right|^2-1)\, dV\,\right|<\epsilon.
\]
\end{proposition}

\begin{proof}Suppose for some $\epsilon>0$, no such $\delta_0$ exists. Then there are manifolds $(M_i, g_i, p_i)$ with distance function $d_i$ satisfying $\mathrm v_{M_i} \geq \mathrm v>0$, and a sequence $\delta_i \to 0$ and points $x_i \in B_{p_i}(r)$ such that 
$B_{x_i}({1}/{\delta_i})$ is $(0, \delta_i^2)$-symmetric but either 
\begin{equation} \label{option:C0}
   \sup_{B_{p_i}(r)}\left| \frac{b_{x_i}^2}{d_{x_i}^2}-1\right|\geq\epsilon 
\end{equation}
or
\begin{equation} \label{option:W12}
   \left|\fint_{B_{p_i}(r)}(\left|\nabla b_{x_i} \right|^2-1)\, dV_i \right| \geq\epsilon. 
\end{equation}
Here $b_{x_i}, d_{x_i}, dV_i$ denote the appropriate functions and the volume measure on $M_i$ with respect to the metrics $g_i$. 
Now by Gromov compactness, a subsequence of $(M_i, g_i, x_i, dV_{i})$ converges in the pointed Gromov-Hausdorff sense to a limit space $ (M_\infty,d_\infty, x_\infty, \mu_{\infty})$ and moreover $\lim_{i \to \infty}\mathrm v_{M_i}$ exists and is strictly positive. Since each $B_{x_i}(1/\delta_i)$ is $(0,\delta_i^2)$-symmetric, the limit space is a cone. Denote $M_\infty = C(Y)$, the cone over a metric measure space $Y$.

Let $G_i$ denote the Green function on $(M_i,  g_i)$ with pole at $x_i$. By \cite[Corollary 4.22]{ding}, the Green function on $M_\infty$ can be explicitly calculated as $G^D_{x_\infty} = (n-2)^{-1}\mu_Y(Y)^{-1}  d_\infty^{2-n}(x_\infty,\cdot)$, where the superscript $D$ signifies the scaling given in Section \ref{sec:greens_scaling} and $\mu_Y$ is the cross sectional measure on $Y$. Renormalizing to $G^{CM}_{x_\infty}$ as discussed in Section \ref{sec:greens_scaling}, 
 we get that $(b_{M_\infty})_{x_\infty}= d_\infty(x_\infty,\cdot)$.  Now \cite[Theorem 3.21]{ding} implies that for any compact set in $M_\infty$, $ b_{x_i}$ converges to $d_{\infty}(x_\infty,\cdot)$ uniformly. Therefore (\ref{option:C0}) can only hold finitely often.

 Now note that (\ref{eq:bgradbound}) implies that there exists an $I$ such that for all $i \geq I$, 
 \[
 \fint_{\Brpi}\left(|b_{x_i}|^2 + |\nabla b_{x_i}|^2 \right) \,dV_i\leq C(r,n,\mathrm v).
 \]
 Moreover, also by (\ref{eq:bgradbound}),
 \[
 \fint_{\Brpi}|\Delta b_{x_i}|^2 \,dV_i= \fint_{\Brpi} \frac{(n-1)^2}{b_{x_i}^2}|\nabla b_{x_i}|^4\,dV_i\leq C(n,\mathrm v)\fint_{\Brpi} \frac{1}{b_{x_i}^2}\,dV_i.
 \]Finally, since \eqref{option:C0} holds only finitely often, for possibly larger $I$, $b_{x_i}^{-2} \leq 2d_{x_i}^{-2}$ on $B_{p_i}(r)$ for all $i \geq I$. So the integral above is finite and bounded by $C(r,n, \mathrm v)$. Let $\lim_i p_i \to p_\infty$,  taking another subsequence if needed. Then, by \cite[Theorem 4.29]{cjn}, $ b_{x_i} \to  d_\infty(x_\infty, \cdot)$ in $W^{1,2}$ on $B_{p_\infty}^{ d_\infty}(r)$.
This limit contradicts (\ref{option:W12}) holding for arbitrarily large $i$, and the result follows.
\end{proof}

Combining Remark \ref{rem:symmetry} with Proposition \ref{prop:UniformEst}, we immediately see that pinching implies uniform estimates.

\begin{corollary}\label{cor:uniform est}
Let $r>0$ be arbitrary. Given $\epsilon>0$, there exists $\delta_0=\delta_0(\epsilon, n, \mathrm v)>0$ such that if $x \in B_p(r)$ satisfies $\mathcal F_{2\delta^{-1},4\delta^{-1}}(x)<\delta^{4+4\mu}/C$, where $0<\delta<\delta_0$ and $\mu,C$ are the constants from Corollary \ref{cor:symmetry}, then 
\[
\sup_{\Brp}\left| \frac{b_{x}^2}{d_{x}^2}-1\right|<\epsilon
\]and 
\[
\left|\fint_{\Brp}(\left| \nabla b_{x} \right|^2-1) \, dV \,\right|<\epsilon.
\]
\end{corollary}

\subsection{Quantitative bounds using $\mathcal F$}\label{sec:Fpinch}

The following proposition provides a quantitative bound on the average $L^2$-norm of the trace-free Hessian of the function $b^2_x$ on a ball in terms of $\mathcal{F}$. Note that items (1) and (3) of the proposition below would also hold with $\mathcal W_{s,2s}(x)$ replacing $\mathcal F_{s,2s}(x)$ if $x \in B_p(1)$ and the domain $B_p(r)$ was replaced by an annulus $A^b(p,C_1s,C_2s)$ where $C_1,C_2$ depend on $\mathrm v,n$. If, moreover, such annuli were connected then a result comparable to item (2) would hold. We do not write out such a result in detail since it is not strong enough to prove something comparable to Theorem \ref{thm:k splitting}. (See the discussion in the Introduction for more on this failure of control by $\mathcal W$.)

\begin{proposition}\label{prop:1pinching}  Let $r>0$. For all  $s \geq 2b_{\infty}^{-1}r$, $x \in B_r(p) \subset M$ and $c_{r,x}:=\fint_{B_p(r)} |\nabla b_x|^2  \, dV$, the following hold:
\begin{enumerate}
    \item $\fint_{B_p(r)} \Hbx^2 \, dV \leq C(n,\tv) \left(\frac sr\right)^n \mathcal F_{s,2s}(x)$,
    \item $r^{-2}\fint_{B_p(r)}\left||\nabla b_x|^2-c_{r,x}\right|^2  \, dV\leq  C(n,\tv) \fint_{\Brp} \Hbx^2b_x^{-2}\, dV$,
    \item $\sup_{B_p(r)}\left|\nabla b_x^2\right|\leq 4b_\infty^{-2}r$. 
\end{enumerate}
\end{proposition}
\begin{proof}
 Note first that by \eqref{eq:tracefreebound}, with $C=2$, 
\[
 \int_{b_x \le s} \Hbx^2\,dV \leq s^n{2^{n+2}} \mathcal F_{s,2s}(x).
\]Now by Lemma \ref{lem:subball}, $B_p(r) \subset \{b_x \leq s\}$ so
\[
\int_{B_p(r)} \Hbx^2\,dV \leq s^n{2^{n+2}} \mathcal F_{s,2s}(x)
\]or
\[
\fint_{\Brp} \Hbx^2 \,dV\leq {2^{n+2}}\frac{ s^n}{|B_p(r)|} \mathcal F_{s,2s}(x)\leq C(n,\tv) \left(\frac sr\right)^n\mathcal F_{s,2s}(x).
\]
This proves item (1).

To prove item (2), observe first that
    \begin{align*}
        \nabla (|\nabla b_x|^2) &= \nabla \left(\frac {|\nabla b_x^2|^2}{4b_x^2}\right)\\
        &=\frac{\mathrm{Hess}_{b_x^2}(\nabla b_x^2, \cdot)}{2b_x^2}-\frac{|\nabla b_x^2|^2 \nabla b_x}{2b_x^3}\\
&=\mathrm{Hess}_{b_x^2}\left(\frac{\nabla b_x}{b_x}, \cdot\right) -\frac{\Delta b_x^2}n g\left( \frac{\nabla b_x}{b_x}, \cdot\right).
    \end{align*}
Therefore, \[|\nabla (|\nabla b_x|^2)|^2 \le \left|\mathrm{Hess}_{b_x^2}-\frac{\Delta b_x^2}{n}g\right|^2 \frac{|\nabla b_x|^2}{b_x^2}.\]

Now (\ref{eq:bgradbound}) and Lemma \ref{lem:NPi} together imply that 
\begin{align*}
\fint_{B_p(r)} \left||\nabla b_x|^2 - c_{r,x}\right|^2 \, dV 
&\leq C(n)r^2  \fint_{B_p(r)} \Hbx^2 \frac{|\nabla b_x|^2}{b_x^2}\, dV \\
&\leq b_\infty^{-2} C(n)r^2 \fint_{B_p(r)} \Hbx^2b_x^{-2}\,dV.
\end{align*}
Finally, item (3) is immediate from (\ref{eq:bgradbound}), as 
\[
|\nabla b_x^2|= 2b_x |\nabla b_x| \leq  2 b_\infty^{-2}d_x \leq 4b_{\infty}^{-2} r
\]
on $B_p(r)$.
\end{proof}

\section{Proofs of the Main Theorems} \label{sec: splitting}

In this section we present the proofs of our main results on the construction of the almost $k$-splitting functions using the Green function. The following is the precise version of Theorem \ref{thm:k splitting intro}.

\begin{theorem}\label{thm:k splitting}
Let $r, \alpha >0$.  There exist $\delta_0=\delta_0(n,\mathrm v, \alpha, r)>0$ and $C=C(n,\mathrm v,\alpha)>0$ such that if $\{x_0, x_1, \dots, x_k\} \subset B_p(r)$ are $(k, \alpha)$-independent and each $B_{x_i}(1/\delta)$ is $(0, \delta^2)$-symmetric, where $0<\delta<\delta_0$, 
then there exists a function $u:\Brp \to \R^k$ 
such that for each $i,j = 1,\cdots,k$,

 \begin{enumerate}[(1)]
     \item \label{item:Hess} \begin{align*}
 r^2\fint_{\Brp}& |\text{Hess}_{u_j}|^2 \,dV \leq\, C\frac{1}{(\delta r)^{n}}\left( \mathcal F_{\delta^{-1},2\delta^{-1}}(x_j) + \mathcal F_{\delta^{-1},2\delta^{-1}}(x_0)\right) \\
 &+Cr^2\left(\fint_{\Brp} \left| \text{Hess}_{b_{x_j}^2} - \frac{\Delta b_{x_j}^2}{n}g\right|^2 b_{x_j}^{-2} \, dV + \fint_{\Brp} \left| \text{Hess}_{b_{x_0}^2} - \frac{\Delta b_{x_0}^2}{n}g\right|^2 b_{x_0}^{-2} \, dV\right),
     \end{align*} 
     \item \label{item:intgrad} $\fint_{\Brp}\left|\langle \nabla u_i, \nabla u_j\rangle - \delta_{ij}\right|^2 \,dV\leq C \sum_{\ell=0}^{k} \mathcal F_{\delta^{-1},2\delta^{-1}}(x_\ell)$, 
     \item \label{item:pointwise grad} $\sup_{\Brp}|\nabla u_i|\leq C$.
 \end{enumerate}
\end{theorem}

\begin{proof} Begin with $\delta_0$ the lesser of that from Proposition \ref{prop:UniformEst}, with $\epsilon < \frac 12$, and $b_\infty / 2r$; thus at this point $\delta_0$ depends only on $n,\mathrm v, r,$ and $\epsilon$. As in Proposition \ref{prop:1pinching}, let $c_{r,x}:=\fint_{\Brp} |\nabla b_x|^2  \, dV$ for each point $x$. 

Consider the functions
\[
\tilde u_j :=\frac{ {\tilde b_{x_j}}^2-{\tilde b_{x_0}}^2-d_{x_0}(x_j)^2}{2d_{x_0}(x_j)}
\]where ${\tilde b_x}^2:=b_x^2/c_{r,x}$ for each point $x$.  By a direct calculation, Proposition \ref{prop:1pinching} item (1), and the fact that $\Delta b_{x_j}^2 = 2n|\nabla b_{x_j}|^2$, there exists a constant $C=C(n,\mathrm v,\alpha)>0$ such that
\begin{align*}
    \fint_{\Brp} |\text{Hess}_{\tilde u_j}|^2 \,dV& \leq \frac 1{4d_{x_0}(x_j)^2}\fint_{\Brp}\left|\frac 1{c_{r,x_j}}\text{Hess}_{ b_{x_j}^2}-\frac 1{c_{r,x_0}}\text{Hess}_{ b_{x_0}^2}\right|^2 \,dV\\
    & \leq \left(\frac 1{\delta r}\right)^n \frac {C}{d_{x_0}(x_j)^2}\left(\frac {\mathcal F_{\delta^{-1},2\delta^{-1}}(x_j)}{c_{r,x_j}^2}+ \frac{\mathcal F_{\delta^{-1},2\delta^{-1}}(x_0)}{c_{r,x_0}^2}\right) +\\& \quad \quad\quad\quad\quad +\frac {C}{d_{x_0}(x_j)^2}\left(\fint_{\Brp}\left| \frac {|\nabla b_{x_j}|^2}{c_{r,x_j}} - \frac {|\nabla b_{x_0}|^2}{c_{r,x_0}}\right|^2\,dV\right).
\end{align*}
By Proposition \ref{prop:UniformEst} and our choice of $\epsilon < \frac 12$, $c_{r,x_j}>1/2$ for each $j=0, \dots, k$. Further, by Proposition \ref{prop:1pinching} item (2), 
\[
\fint_{\Brp}\left| \frac {|\nabla b_{x_j}|^2}{c_{r,x_j}}-1 \right|^2 \, dV\leq r^2C(n,\mathrm v) \fint_{\Brp} \Hbx^2b_x^{-2} \, dV.
\]
Since $\{x_0, \cdots, x_k\}$ are $(k, \alpha)$-independent there is a lower bound $d_{x_0}(x_j) \ge C(\alpha)r$ for each $j=1 \dots k$. Taken together we conclude that
\begin{align*}
 \fint_{\Brp} & |\text{Hess}_{\tilde u_j}|^2 \,dV \leq\, C\frac{1}{\delta^n r^{n+2}}\left( \mathcal F_{\delta^{-1},2\delta^{-1}}(x_j) + \mathcal F_{\delta^{-1},2\delta^{-1}}(x_0)\right) \\
& +C\left(\fint_{\Brp} \left| \text{Hess}_{b_j^2} - \frac{\Delta b_j^2}{n}g\right|^2 b_{x_j}^{-2} \, dV + \fint_{\Brp} \left| \text{Hess}_{b_0^2} - \frac{\Delta b_0^2}{n}g\right|^2 b_{x_0}^{-2} \, dV\right),
\end{align*}
which is exactly item (\ref{item:Hess}).

As for item (\ref{item:pointwise grad}), we apply Proposition \ref{prop:1pinching} item (3) to get
\begin{equation}\label{eq:tildegradbound}
\sup_{B_p(r)}|\nabla \tilde{u}_j| \le \sup_{B_p(r)}\frac{|\nabla \tilde{b}^2_{x_j}| + |\nabla \tilde{b}^2_{x_0}|}{2d_{x_0}(x_j)} \le \frac{4b_\infty^{-2} r}{C(\alpha)r}\left(\frac 1{c_{r,x_0} }+ \frac 1{c_{r, x_j}} \right)\le C( n,\mathrm{v},\alpha).
\end{equation}

So we now have functions $\tilde u_i:B_r(p) \to \mathbb R$ satisfying items (\ref{item:Hess}) and (\ref{item:pointwise grad}). Also note by the $C^0$ estimate of Proposition \ref{prop:UniformEst} and the triangle inequality,
\begin{equation}\label{eq:olditem3}
\sup_{\Brp}\left|\tilde u_j-\frac{d_{x_j}^2 - d_{x_0}^2-d_{x_0}(x_j)^2}{2d_{x_0}(x_j)}\right|
\leq C(\alpha)\epsilon r.
\end{equation}

To get item (\ref{item:intgrad}) we follow the proof of \cite[Lemma 6.7]{cjn}. We claim that there exists a $k\times k$ lower triangular matrix $A$ where $|A| \leq C(n, \mathrm v, \alpha)$ and such that for
 $(u_1, \dots, u_k):=A(\tilde u_1, \dots, \tilde u_k)$,
 \[
 \fint_{\Brp}\langle \nabla u_j, \nabla u_\ell\rangle \, dV = \delta_{j\ell}.
 \]
As in \cite[Proposition 6.7]{cjn}, we proceed by contradiction. Hence assume there exist sequences of Riemannian manifolds $(M^n_i, g_i, p_i)$, with $\mathrm{v}_{M_i} \geq \mathrm{v}>0$, $\text{Ric}_{M_i} \geq 0$, $n \geq 3$, and $\{x_0^i, \dots, x_k^i\}\subset B^{d_i}_{p_i}(r)$ a collection of $(k,\alpha)$-independent points so that $B_{x^i_j}(\delta_i^{-1})$ are all $(0,\delta_i^2)$-symmetric where $\delta_i \to 0$,  and have uniform gradient bounds given by \eqref{eq:tildegradbound}, but either there do not exist $k\times k$ lower triangular matrices $A_i$ such that 
\[
(u_1^i, \dots, u_k^i):=A_i(\tilde u_1^i, \dots, \tilde u_k^i)
\]satisfy the estimate
\[
\fint_{B_r(p_i)} \langle \nabla u_j^i, \nabla u_\ell^i \rangle \, dV_i= \delta_{j\ell} 
\] or if such matrices exist then $|A_i| \to \infty$. Note that we may choose this sequence in such a way that the  functions $\tilde u_j^i$, $j=1, \dots, k$, satisfy \eqref{eq:olditem3} with $\epsilon$ replaced by $\epsilon_i \to 0$.

Letting $(M_i, d_i, p_i,\mu_{g_i}) \to (M_\infty, d_\infty, p_\infty, \mu_\infty)$ in the pointed measured Gromov-Hausdorff sense, we further have $x_j^i \to x_j^\infty \in \overline {B_r(p_\infty)}$ for $j=0, \dots, k$. By the standard splitting theorem, $M_\infty=\R^k \times C(Y)$ for some cone $C(Y)$ and $x_j^\infty \in \R^k \times \{o\}$ where $o$ is the cone point of $C(Y)$. As the $\tilde u_j^i$ converge in $C^0$ uniformly, we observe that
\[
\tilde u_j^i \to \tilde u_j^\infty:=\frac{ {d_\infty^2({x_j^\infty}, \cdot)}-{d_\infty^2({x_0^\infty}, \cdot)}-d_\infty^2({x_0^\infty},x_j^\infty)}{2d_\infty({x_0^\infty},x_j^\infty)} \text{ on }B_{p_\infty}(r)
\]which is a linear function. Since $\tilde u_1^\infty, \dots, \tilde u_k^\infty$ form a basis for $\R^k$ there exists a lower triangular matrix $A_\infty=A_\infty(n,\mathrm v,\alpha)$ such that for
\[
(u_1^\infty, \dots, u_k^\infty):= A_\infty(\tilde u_1^\infty, \dots, \tilde u_k^\infty),
\]it holds that
\[
\fint_{B_r(p_\infty)}\langle \nabla u_j^\infty, \nabla u_\ell^\infty \rangle\, dV_\infty = \delta_{j\ell}.
\]

We can appeal to the $W^{1,2}$-convergence on balls from \cite[Proposition 4.29]{cjn} and complete the contradiction. Namely, there exist lower triangular matrices $\hat A_i$ with $|\hat A_i - A_\infty| \to 0$ such that the functions $(\hat u_1^i, \dots, \hat u_k^i):= \hat A_i(\tilde u_1^i, \dots, \tilde u_k^i)$ satisfy the orthogonality conditions on $B_{p_i}(r)$. Since $|A_\infty|$ has bounded norm, this implies a contradiction. This proves item \eqref{item:intgrad}.

Also note that the linear and bounded change introduced by $A$ does not destroy items \eqref{item:Hess} and \eqref{item:pointwise grad}.

\end{proof}

\begin{remark} \label{rmk: not harmonic}
In addition to satisfying $L^2$ bounds on the Hessian and gradient as in items (\ref{item:Hess}) and (\ref{item:intgrad}), almost $k$-splitting maps arising in the literature are often harmonic by construction. For instance in \cite[Proposition 6.4]{cjn}, the $k$-splitting functions are defined using the Poisson solutions depending on the domain, not the Green function on the manifold. 

While the almost splitting functions we consider are quite natural objects defined using the Green functions centered at the points $x_i$, they do not need to be harmonic. Indeed, in general the values of $|\Delta u_i|$ are bounded by a constant depending on $n$ and $\tv$ but are not necessarily zero. 

Note that item (\ref{item:intgrad}) implies that $|\nabla u_i| \leq 1+ \epsilon$ on average when the pinching of $\mathcal F$ is small. If $u_i$ is harmonic, this can be improved to a pointwise bound as in \cite[Lemma 3.34]{chen2}. The improvement follows when the term $\langle \nabla \Delta u_i, \nabla u_i\rangle$ involved in the Bochner formula is 0 or small in an integral sense (see \cite[Lemma 4.3--4.4]{hp}). Therefore, attaining higher derivative control on $u_i$ would allow us to conclude a similar pointwise result.
\end{remark}
Corollary \ref{cor:symmetry} and Remark \ref{rem:symmetry} immediately imply that the result of Theorem \ref{thm:k splitting} holds if we replace the hypothesis that $B_{x_i}(1/\delta)$ is $(0,\delta^2)$-symmetric by a hypothesis purely on the pinching.

\begin{corollary}\label{cor:splitting estimates}
Let $r, \alpha >0$.  There exist $\delta_0=\delta_0(n,\mathrm v, \alpha,r)>0$ and $C=C(n,\mathrm v,\alpha)>0$ such that if $\{x_0, x_1, \dots, x_k\} \subset B_p(r)$ are $(k, \alpha)$-independent and $\mathcal F_{2\delta^{-1},4\delta^{-1}}(x_i)<\delta^{4+4\mu}/C$, where $0<\delta<\delta_0$ and $\mu,C$ are the constants from Corollary \ref{cor:symmetry}, then the result of Theorem \ref{thm:k splitting} holds.
\end{corollary}

Observe that if the integral of the trace-free Hessian on the right hand side of item (\ref{item:Hess}) of Theorem \ref{thm:k splitting} had not been weighted with $b^{-2}$, one could immediately apply Proposition \ref{prop:0pinching_2} to obtain an estimate of the $k$-splitting map purely by the pinching $\mathcal F$. In the following proposition, we note that if the pinching at large scales is sufficiently small, relative to both its scale and the supremum of the trace-free Hessian on a much smaller ball, then we can bound the second term on the right hand side of item \eqref{item:Hess} in Theorem \ref{thm:k splitting} by the pinching.

\begin{proposition}\label{prop:pinching bound} 
 For $p \in M$ and $r>0$, let $H:= \sup_{x \in B_p(2r)} \left|\text{Hess}_{b_x^2} - \frac{\Delta b_x^2}{n}g\right|^2(x)$. There exists $\eta_0=\eta_0(H,n)>0$ such that for all $s > 2b_\infty^{-1}r$, if $x \in \Brp$ is such that $\left( \frac{s}{r}\right)^n\mathcal F_{s,2s}(x)<\eta_0$, then
\[s^2\fint_{\Brp} \left| \text{Hess}_{b_x^2} - \frac{\Delta b_x^2}{n}g\right|^2 b_x^{-2} \, dV \le C(H,n,\tv) \left(\frac sr\right)^{n}  \mathcal F_{s,2s}(x)^{1-2/n}.\]  
\end{proposition}

\begin{proof} 
Fix $\eta>0$, to be chosen sufficiently small later, so that $B_x(\eta) \subset B_p(2r)$.
Since $G$ is harmonic and strictly positive on $B_p(2r) \setminus B_x(\eta)$, by the maximum principle for the harmonic function $G$,
\[ 
\sup_{B_p(2r) \setminus B_x(\eta)}b_x^{-2} \le \sup_{\partial B_x(\eta)} b_x^{-2} \le \eta^{-2} \sup_{\partial B_x(\eta)}\frac{d_x^{2}}{b_x^2}.
\]Noting that $B_p(r) \subset B_x(2r)\subset \{b_x \leq 2b_\infty^{-1}r\}$, where the second containment follows from \eqref{eq:bgradbound}, Proposition \ref{prop:0pinching_2} implies that \begin{align*}
 \int_{B_p(r) \setminus B_x(\eta)} \left| \text{Hess}_{b_x^2} - \frac{\Delta b_x^2}{n}g\right|^2 b_x^{-2} \, dV &\le \eta^{-2} \sup_{\partial B_x(\eta)}\frac{d_x^{2}}{b_x^2} \int_{B_x(2r)} \left| \text{Hess}_{b_x^2} - \frac{\Delta b_x^2}{n}g\right|^2\, dV \\
 &\le C(n)\eta^{-2}s^n\left( \sup_{\partial B_x(\eta)}\frac{d_x^{2}}{b_x^2}\right)  \mathcal F_{s,2s}(x).   
\end{align*}
Further, observe that 
\begin{align*}
\int_{B_x(\eta)}\left| \text{Hess}_{b_x^2} - \frac{\Delta b_x^2}{n}g\right|^2 b_x^{-2} \, dV &\le \sup_{B_x(\eta)} \left|\frac{d_x}{b_x}\left(\text{Hess}_{b_x^2} - \frac{\Delta b_x^2}{n}g\right)\right|^2 \int_{B_x(\eta)}d_x^{-2} dV \\
&\le C(\tv)\sup_{B_x(\eta)} \left|\frac{d_x}{b_x}\left(\text{Hess}_{b_x^2} - \frac{\Delta b_x^2}{n}g\right)\right|^2 \int_0^\eta s^{-2} s^{n-1}ds \\
&\le {C(n,\tv)}\sup_{B_x(\eta)}\left|\frac{d_x}{b_x}\left(\text{Hess}_{b_x^2} - \frac{\Delta b_x^2}{n}g\right)\right|^2\eta^{n-2} \\
&\le  {C(n,\tv)} \left(\sup_{B_x(\eta)}\frac{d_x^2}{b_x^2}\right) 
H\eta^{n-2} .
\end{align*}
Summing up and averaging, we have that
\begin{align*}
\fint_{\Brp} \left| \text{Hess}_{b_x^2} - \frac{\Delta b_x^2}{n}g\right|^2 b_x^{-2} \, dV \le & C(n,\tv)\left(\sup_{B_x(\eta)}\frac{d_x^{2}}{b_x^2}\right)\eta^{-2}\left( \frac{s^n}{r^n}\mathcal F_{s,2s}(x) + \frac{\eta^n}{r^n}H \right)
\end{align*}
for any $0<\eta<2r-d_x(p)$.  By \cite{msy} (also see \cite[Proposition 1.15]{cm-schroedinger}), there exists a uniform constant $C'=C'(n, \tv) \ge 1$, so that for every $x,y \in M,$ 
\begin{equation*}\label{eq:C'def}
\frac{d_x(y)}{C'}\leq b_x(y).
\end{equation*}It follows that
\[
\sup_{B_x(\eta)}\frac{d_x^{2}}{b_x^2}\leq C(n,\tv)
\]for any $\eta>0$. Thus, for any $\eta<r$,
\[
\fint_{\Brp} \left| \text{Hess}_{b_x^2} - \frac{\Delta b_x^2}{n}g\right|^2 b_x^{-2} \, dV 
\le C(n,\tv) \eta^{-2}\left( \frac{s^n}{r^n}\mathcal F_{s,2s}(x) + \frac{\eta^n}{r^n}H \right).
\]
Optimizing the right hand side for $\eta$, we observe that if $\eta_0$ is small enough so that $C(n)\frac{\mathcal F_{s,2s}(x)}Hs^n <r^n$ then we may substitute for the optimal $\eta=\left(\frac{2s^n \mathcal{F}_{s,2s}(x)}{(n-2)H} \right)^{1/n}$ and arrive at the estimate
\[
s^2\fint_{\Brp} \left| \text{Hess}_{b_x^2} - \frac{\Delta b_x^2}{n}g\right|^2 b_x^{-2} \, dV \le C(H,n,\tv) \frac{s^n}{r^n} \mathcal F_{s,2s}(x)^{1-2/n}.
\]
\end{proof}

Combining the proposition and Theorem \ref{thm:k splitting}, we immediately get the following precise version of Theorem \ref{thm:k splitting by pinching only intro}.

\begin{theorem}\label{thm:k splitting by pinching only}
For $p \in M$ and $r>0$, let $H:= \sup_{x \in B_p(2r)} \left|\text{Hess}_{b_x^2} - \frac{\Delta b_x^2}{n}g\right|^2(x)$. Let $\alpha >0$.  There exist $\delta_0=\delta_0(n,  \mathrm v, \alpha,r)>0$, $C=C(n,\mathrm v, \alpha)>0$, and $\eta_0=\eta_0(H,n)>0$ such that if $\{x_0, x_1, \dots, x_k\} \subset B_p(r)$ are $(k, \alpha)$-independent and, for each $i=0\dots k$, $B_{x_i}(1/\delta)$ is $(0,\delta^2)$-symmetric where $0<\delta<\delta_0$, and moreover $\left(\frac{1}{r\delta } \right)^n\mathcal F_{\delta^{-1},2\delta^{-1}}(x_i)<\eta_0$,
then there exists a function $u:\Brp \to \R^k$ 
such that for each $i,j = 1,\cdots,k$,

 \begin{enumerate}[(1)]
     \item \begin{align*}r^2\fint_{\Brp}&|\text{Hess}_{u_j}|^2 \,dV  \\& \leq C\frac{1}{(\delta r)^{n}}\left( \mathcal F_{\delta^{-1},2\delta^{-1}}(x_j) + \mathcal F_{\delta^{-1},2\delta^{-1}}(x_0)+(\delta r)^2\left( \mathcal F_{\delta^{-1},2\delta^{-1}}(x_j)^{1-2/n} + \mathcal F_{\delta^{-1},2\delta^{-1}}(x_0)^{1-2/n}\right)\right) ,\end{align*}
     \item $\fint_{\Brp}\left|\langle \nabla u_i, \nabla u_j\rangle - \delta_{ij}\right|^2 \,dV\leq C \sum_{\ell=0}^{k} \mathcal F_{\delta^{-1},2\delta^{-1}}(x_\ell)$, 
     \item $\sup_{\Brp}|\nabla u_i|\leq C$.
 \end{enumerate}    
\end{theorem}

\begin{remark} 
Once again note that we can replace the $(0,\delta^2)$-symmetry hypothesis by the assumption that $\mathcal{F}_{2\delta^{-1},4\delta^{-1}}(x_i) \leq \delta^{4+4\mu}/C$, by Remark \ref{rem:symmetry}.
\end{remark}

Finally, we state the precise rigidity theorem. The proof follows from standard compactness-contradiction results, which we leave to the reader.

\begin{theorem}\label{thm:rigidity}
 For $p \in M$ and $r>0$, let $H:= \sup_{x \in B_p(2r)} \left|\text{Hess}_{b_x^2} - \frac{\Delta b_x^2}{n}g\right|^2(x)$. Let $\alpha >0$.  Given $\epsilon>0$ there exist $\delta_0=\delta_0( \epsilon,n,  \mathrm v, \alpha,r)>0$, $C=C(n,\mathrm v, \alpha)>0$, and $\eta_0=\eta_0(H,n)>0$ such that if $\{x_0, x_1, \dots, x_k\} \subset B_p(r)$ are $(k, \alpha)$-independent and, for each $i=0\dots k$, $B_{x_i}(1/\delta)$ is $(0,\delta^2)$-symmetric where $0<\delta<\delta_0$, and moreover $\left(\frac{1}{r\delta} \right)^n\mathcal F_{\delta^{-1},2\delta^{-1}}(x_i)<\eta_0$,   then there exist $C(Y)$, the cone over a metric measure space $Y$, and a function $\mathcal U:=(u,\tilde u): B_p(r) \to \R^k \times C(Y)$, where $u:B_p(r) \to \R^k$ is the function from Theorem \ref{thm:k splitting by pinching only}, such that  
\[
d_{GH}(B_p({r/2}), \mathcal U(B_p({r/2})))<\epsilon r.
\]
\end{theorem}

\end{document}